\documentclass{amsart}
\usepackage{amsfonts}

\newtheorem{theorem}{Theorem}
\theoremstyle{plain}
\newtheorem{acknowledgement}{Acknowledgement}

\newtheorem{corollary}{Corollary}

\newtheorem{definition}{Definition}
\newtheorem{example}{Example}

\newtheorem{lemma}{Lemma}

\newtheorem{proposition}{Proposition}
\newtheorem{remark}{Remark}

\numberwithin{equation}{section}
\input{tcilatex}

\begin{document}
\title[On the Jensen functional and superquadraticity]{On the Jensen
functional and superquadraticity}
\author{Flavia-Corina Mitroi-Symeonidis}
\address{LUMINA - University of South-East Europe, Faculty of Engineering
Sciences, \c{S}os. Colentina 64b, Sector 2, RO-021187, Bucharest, Romania }
\email{fcmitroi@yahoo.com}
\author{Nicu\c{s}or Minculete}
\address{Transilvania University of Bra\c{s}ov, Iuliu Maniu Street, No. 50,
500091, Bra\c{s}ov, Romania}
\email{minculeten@yahoo.com}
\date{}

\begin{abstract}
In this note we give a recipe which describes upper and lower bounds for the
Jensen functional under superquadraticity conditions. Some results involve
the Chebychev functional. We give a more general definition of these
functionals and establish the analogue results.
\end{abstract}

\maketitle

\section{Introduction}

The object of this paper is to derive some results related to the Jensen
functional\emph{\ }in the framework of superquadratic functions. We are
interested in finding bounds both for the discrete and continuous case.

For the reader's convenience, let us briefly state known facts regarding the
principal tools, the superquadraticity and the Jensen functional. See S.
Abramovich and S. S. Dragomir \cite{abr09} for details and proofs.

\begin{definition}[{\protect\cite[Definition 2.1]{abr04}}]
A function $f$ defined on an interval $I=[0,a]$ or $[0,\infty )$ is \emph{%
superquadratic} if for each $x$ in $I$ there exists a real number $C(x)$
such that 
\begin{equation}
f(y)-f(x)\geq f(|y-x|)+C(x)(y-x)  \label{SQ}
\end{equation}%
for all $y\in I$.
\end{definition}

We say that $f$ is a \emph{subquadratic function }if $-f$ is superquadratic.
The set of superquadratic functions is closed under addition and positive
scalar multiplication.

\begin{example}[\protect\cite{abrPP}]
The function $f\left( x\right) =x^{p},$ $p\geq 2\ $is superquadratic with $%
C(x)=f^{\prime }\left( x\right) =px^{p-1}.$ Similarly, $g\left( x\right)
=-\left( 1+x^{1/p}\right) ^{p},$ $p>0\ $is superquadratic with $C(x)=0.$
Also $h(x)=x^{2}\log x$ with $C(x)=h^{\prime }(x)=x(2\log x+1)$ is a
superquadratic function (but not monotonic and not convex).
\end{example}

\begin{example}[\protect\cite{Wal14}]
Some elementary functions are not superquadratic, such as $f\left( x\right)
=x$ and $f\left( x\right) =\exp x.$
\end{example}

\begin{lemma}[{\protect\cite[Lemma 2.2]{abr04}}]
Let $f$ be a superquadratic function with $C(x)$ defined as above.

(i) Then $f(0)$ $\leq 0$.

(ii) If $f(0)=f^{\prime }(0)=0$, then $C(x)=f^{\prime }(x),$ whenever $f$ is
differentiable at $x>0$.

(iii) If $f\geq 0$, then $f$ is convex and $f(0)=f^{\prime }(0)=0.$
\end{lemma}

\begin{theorem}[{\protect\cite[Theorem 2.3]{abr04}}]
\label{th_integral}The inequality 
\begin{equation}
f\left( \int \varphi \mathrm{d}\mu \right) \leq \int f\left( \varphi \left(
s\right) \right) -f\left( \left\vert \varphi \left( s\right) -\int \varphi 
\mathrm{d}\mu \right\vert \right) \mathrm{d}\mu \left( s\right)
\end{equation}%
holds for all probability measures $\mu $ and all nonnegative, $\mu $%
-integrable functions $\varphi $ if and only if $f$ is superquadratic.
\end{theorem}

\begin{definition}[\protect\cite{abr09}]
\label{main}Let $f$ be a\ real valued function defined on an interval $I$, $%
x_{1},...,x_{n}\in I$ and $p_{1},...,p_{n}\in \left( 0,1\right) $ such that $%
\sum_{i=1}^{n}p_{i}=1$. The \emph{Jensen functional} is defined by 
\begin{equation}
\mathcal{J}\left( f,\mathbf{p},\mathbf{x}\right) =\sum_{i=1}^{n}p_{i}f\left(
x_{i}\right) -f\left( \sum_{i=1}^{n}p_{i}x_{i}\right)
\end{equation}%
and the \emph{Chebychev functional} is defined by 
\begin{equation}
\mathcal{T}\left( f,\mathbf{p},\mathbf{x}\right) =\sum_{i=1}^{n}p_{i}\left(
x_{i}-\sum_{j=1}^{n}p_{j}x_{j}\right) f\left( x_{i}\right) .
\end{equation}
\end{definition}

See \cite{dra06} and \cite{nic99}.

The discrete form of Theorem \ref{th_integral} is as follows.

\begin{proposition}[{\protect\cite[Lemma 2]{abr09}}]
\label{6}Let $x_{i}\geq 0,$ $i=1,...,n,$ and $p_{i}>0,$ $i=1,...,n,$ with $%
\sum_{i=1}^{n}p_{i}=1.$ If $f$ is superquadratic, then 
\begin{equation*}
\mathcal{J}\left( f,\mathbf{p},\mathbf{x}\right) \geq
\sum_{i=1}^{n}p_{i}f\left( \left\vert
x_{i}-\sum_{j=1}^{n}p_{j}x_{j}\right\vert \right) .
\end{equation*}
\end{proposition}

\begin{theorem}[{\protect\cite[Theorem 4]{abr09}}]
\label{5}Assume that $x_{i}\in I,$ $i=1,...,n,$ $p_{i}>0$ are such that $%
\sum_{i=1}^{n}p_{i}=1$ and $r_{i}>0$ are such that $\sum_{i=1}^{n}r_{i}=1.$
We denote 
\begin{equation*}
m=\min_{i=1,...,n}\left\{ \frac{p_{i}}{r_{i}}\right\} ,\text{ }%
M=\max_{i=1,...,n}\left\{ \frac{p_{i}}{r_{i}}\right\} .
\end{equation*}%
Then 
\begin{eqnarray*}
&&\mathcal{J}\left( f,\mathbf{p},\mathbf{x}\right) -m\mathcal{J}\left( f,%
\mathbf{r},\mathbf{x}\right) \\
&\geq &mf\left( \left\vert \sum_{i=1}^{n}\left( r_{i}-p_{i}\right)
x_{i}\right\vert \right) +\sum_{i=1}^{n}\left( p_{i}-mr_{i}\right) f\left(
\left\vert x_{i}-\sum_{j=1}^{n}p_{j}x_{j}\right\vert \right)
\end{eqnarray*}%
and 
\begin{eqnarray*}
&&M\mathcal{J}\left( f,\mathbf{r},\mathbf{x}\right) -\mathcal{J}\left( f,%
\mathbf{p},\mathbf{x}\right) \\
&\geq &f\left( \left\vert \sum_{i=1}^{n}\left( r_{i}-p_{i}\right)
x_{i}\right\vert \right) +\sum_{i=1}^{n}\left( Mr_{i}-p_{i}\right) f\left(
\left\vert x_{i}-\sum_{j=1}^{n}r_{j}x_{j}\right\vert \right)
\end{eqnarray*}%
for all superquadratic functions $f$.
\end{theorem}

\begin{definition}
\label{2}Assume that $\mathbf{x}=\left( x_{1},...,x_{n}\right) \in I^{n},$ $%
\mathbf{p}=\left( p_{1},...,p_{n}\right) \ $are such that $p_{i}>0$, $%
\sum_{i=1}^{n}p_{i}=1,$ $\mathbf{q}=\left( q_{1},...,q_{k}\right) $ are such
that $q_{i}>0$, $\sum_{i=1}^{k}q_{i}=1$ ($1\leq k\leq n$). We define%
\begin{equation*}
\mathcal{J}_{k}\left( f,\mathbf{p},\mathbf{q},\mathbf{x}\right)
:=\sum_{i_{1},...,i_{k}=1}^{n}p_{i_{1}}...p_{i_{k}}f\left(
\sum_{j=1}^{k}q_{j}x_{i_{j}}\right) -f\left( \sum_{i=1}^{n}p_{i}x_{i}\right)
.
\end{equation*}
\end{definition}

Obviously $\mathcal{J}_{1}\left( f,\mathbf{p},\mathbf{q},\mathbf{x}\right) =%
\mathcal{J}\left( f,\mathbf{p},\mathbf{x}\right) .$ We quote now some
results that we refine in the following section.

\begin{proposition}[{\protect\cite[Theorem 6]{abr09}}]
\label{4}Let $x_{i}\geq 0,$ $i=1,...,n,$ and $p_{i}>0,$ $i=1,...,n,$ such
that $\sum_{i=1}^{n}p_{i}=1$ and $q_{i}>0$, $i=1,...,k,$ $%
\sum_{i=1}^{k}q_{i}=1$ ($1\leq k\leq n$)$.$ If $f$ is superquadratic, then 
\begin{equation*}
\mathcal{J}_{k}\left( f,\mathbf{p},\mathbf{q},\mathbf{x}\right) \geq
\sum_{i_{1},...,i_{k}=1}^{n}p_{i_{1}}...p_{i_{k}}f\left( \left\vert
\sum_{j=1}^{k}q_{j}x_{i_{j}}-\sum_{j=1}^{n}p_{j}x_{j}\right\vert \right) .
\end{equation*}
\end{proposition}

Notice that the Proposition \ref{4} is a slightly more general assertion
than Proposition \ref{6} above.

\begin{theorem}[{\protect\cite[Theorem 7]{abr09}}]
\label{18}Assume that $\mathbf{x}=\left( x_{1},x_{2},...,x_{n}\right) \in
I^{n},$ $\mathbf{p}=\left( p_{1},p_{2},...,p_{n}\right) \ $are such that $%
p_{i}>0$, $\sum_{i=1}^{n}p_{i}=1,$ $\mathbf{q}=\left(
q_{1},q_{2},...,q_{k}\right) $ are such that $q_{i}>0$, $%
\sum_{i=1}^{k}q_{i}=1$ ($1\leq k\leq n$) and $\mathbf{r}=\left(
r_{1},r_{2},...,r_{n}\right) $ are such that $r_{i}>0$, $%
\sum_{i=1}^{n}r_{i}=1.$ We denote%
\begin{equation*}
m=\min_{1\leq i_{1},...,i_{k}\leq n}\left\{ \frac{p_{i_{1}}...p_{i_{k}}}{%
r_{i_{1}}...r_{i_{k}}}\right\} \text{ and }M=\max_{1\leq i_{1},...,i_{k}\leq
n}\left\{ \frac{p_{i_{1}}...p_{i_{k}}}{r_{i_{1}}...r_{i_{k}}}\right\} .
\end{equation*}%
If $f$ is superquadratic then 
\begin{eqnarray*}
&&\mathcal{J}_{k}\left( f,\mathbf{p},\mathbf{q},\mathbf{x}\right) -m\mathcal{%
J}_{k}\left( f,\mathbf{r},\mathbf{q},\mathbf{x}\right) \\
&\geq &mf\left( \left\vert \sum_{i=1}^{n}\left( r_{i}-p_{i}\right)
x_{i}\right\vert \right) \\
&&+\sum_{i_{1},...,i_{k}=1}^{n}\left(
p_{i_{1}}...p_{i_{k}}-mr_{i_{1}}...r_{i_{k}}\right) f\left( \left\vert
\sum_{j=1}^{k}q_{j}x_{i_{j}}-\sum_{j=1}^{n}p_{j}x_{j}\right\vert \right)
\end{eqnarray*}%
and%
\begin{eqnarray*}
&&M\mathcal{J}_{k}\left( f,\mathbf{r},\mathbf{q},\mathbf{x}\right) -\mathcal{%
J}_{k}\left( f,\mathbf{p},\mathbf{q},\mathbf{x}\right) \\
&\geq &f\left( \left\vert \sum_{i=1}^{n}\left( r_{i}-p_{i}\right)
x_{i}\right\vert \right) \\
&&+\sum_{i_{1},...,i_{k}=1}^{n}\left(
Mr_{i_{1}}...r_{i_{k}}-p_{i_{1}}...p_{i_{k}}\right) f\left( \left\vert
\sum_{j=1}^{k}q_{j}x_{i_{j}}-\sum_{j=1}^{n}r_{j}x_{j}\right\vert \right) .
\end{eqnarray*}
\end{theorem}

Results involving Jensen's and Chebychev's inequalities are sometimes stated
in terms of probability measures rather than summation or Lebesgue
integration. Then some of our results can be derived from the ones above
applied to a product measure, as Gord Sinnamon pointed out during some
useful discussions. \ 

Section 3 contains a definition of such functionals and analogue results.
Their study is done for the discrete and integral case, not in probabilistic
terms. The distinction between summation and integration is not artificial,
but useful for different areas of study as information theory and transport
theory.

For convex, strong convex and superquadratic functions the interested reader
can also find relevant results in \cite{mit10} and \cite{mitCIA10}.\pagebreak

\bigskip

\section{Main results}

\subsection{More on the Jensen functional}

\begin{theorem}
Let $f$ be a superquadratic function defined on an interval $I=\left[ 0,a%
\right] $ or $\left[ 0,\infty \right) $, $x_{1},x_{2},...,x_{n}\in I$ and $%
p_{1},p_{2},...,p_{n}\in \left( 0,1\right) $ such that $%
\sum_{i=1}^{n}p_{i}=1,$ $\lambda \in \left[ 0,1\right] .$ Then 
\begin{equation}
\sum_{i=1}^{n}p_{i}f\left( \left( 1-\lambda \right)
\sum_{j=1}^{n}p_{j}x_{j}+\lambda x_{i}\right) -f\left(
\sum_{j=1}^{n}p_{j}x_{j}\right) \geq \sum_{i=1}^{n}p_{i}f\left( \lambda
\left\vert x_{i}-\sum_{j=1}^{n}p_{j}x_{j}\right\vert \right) .  \label{N}
\end{equation}
\end{theorem}

\begin{proof}
Let $f$ be a superquadratic function with $C(x)$ defined as above. Then
replacing $y$ by $\left( 1-\lambda \right) x+\lambda y$ in (\ref{SQ}) we
deduce the inequality%
\begin{equation*}
f(\left( 1-\lambda \right) x+\lambda y)-f(x)\geq f(\lambda |y-x|)+\lambda
C(x)(y-x).
\end{equation*}%
In this inequality we put $x=\sum_{j=1}^{n}p_{j}x_{j}$ and $y=x_{i}.$
Multiplying by $p_{i}$ and summing over $i$ we get the conclusion.
\end{proof}

For $\lambda =1$ we recover the result of Proposition \ref{6}.

As an immediat consequence of this result, due to the convexity of positive
superquadratic functions, we get the following lower bound of interest:

\begin{corollary}
\label{N1}Let $f\geq 0$ be a superquadratic function defined on an interval $%
I=\left[ 0,a\right] $ or $\left[ 0,\infty \right) $, $x_{1},x_{2},...,x_{n}%
\in I$ and $p_{1},p_{2},...,p_{n}\in \left( 0,1\right) $ such that $%
\sum_{i=1}^{n}p_{i}=1,$ $\lambda \in \left[ 0,1\right] .$ Then 
\begin{equation*}
\mathcal{J}\left( f,\mathbf{p},\mathbf{x}\right) \geq
2\sum_{i=1}^{n}p_{i}f\left( \frac{1}{2}\left\vert
x_{i}-\sum_{j=1}^{n}p_{j}x_{j}\right\vert \right) .
\end{equation*}
\end{corollary}

\begin{proof}
In (\ref{N}) we consider $\lambda =\frac{1}{2}:$%
\begin{equation}
\sum_{i=1}^{n}p_{i}f\left( \frac{\sum_{j=1}^{n}p_{j}x_{j}+x_{i}}{2}\right)
-f\left( \sum_{j=1}^{n}p_{j}x_{j}\right) \geq \sum_{i=1}^{n}p_{i}f\left( 
\frac{1}{2}\left\vert x_{i}-\sum_{j=1}^{n}p_{j}x_{j}\right\vert \right) .
\end{equation}%
Since by Jensen's inequality one has%
\begin{equation*}
\frac{1}{2}\left[ f\left( \sum_{j=1}^{n}p_{j}x_{j}\right) +f\left(
x_{i}\right) \right] \geq f\left( \frac{\sum_{j=1}^{n}p_{j}x_{j}+x_{i}}{2}%
\right) ,
\end{equation*}%
we get 
\begin{equation}
\frac{1}{2}\sum_{i=1}^{n}p_{i}\left[ f\left( \sum_{j=1}^{n}p_{j}x_{j}\right)
+f\left( x_{i}\right) \right] -f\left( \sum_{j=1}^{n}p_{j}x_{j}\right) \geq
\sum_{i=1}^{n}p_{i}f\left( \frac{1}{2}\left\vert
x_{i}-\sum_{j=1}^{n}p_{j}x_{j}\right\vert \right) .
\end{equation}%
This completes the proof.
\end{proof}

The interested reader can refine our last result by using refinements of
Jensen's inequality instead of the classic result.

\subsection{The discrete case}

Motivated by the above results, introduce in a natural other functionals.

\begin{definition}
\label{112}Assume that we have a real valued function $f$ defined on an
interval $I,$ the real numbers $p_{ij},$ $i=1,...,k\ $and $j=1,...,n_{i}$
such that $p_{ij}>0,$ $\sum_{j=1}^{n_{i}}p_{ij}=1$ for all $i=1,...,k\ $(we
denote $\mathbf{p}_{i}=$ $\left( p_{i1},p_{i2},...,p_{in_{i}}\right) $), $%
\mathbf{x}_{i}=$ $\left( x_{i1},x_{i2},...,x_{in_{i}}\right) \in I^{n_{i}}$
for all $i=1,...,k\ $and $\mathbf{q}=\left( q_{1},q_{2},...,q_{k}\right) ,$ $%
q_{i}>0$ such that $\sum_{i=1}^{k}q_{i}=1.$ We define the generalized \emph{%
Jensen functional} by%
\begin{eqnarray*}
\mathcal{J}_{k}\left( f,\mathbf{p}_{1},...,\mathbf{p}_{k},\mathbf{q},\mathbf{%
x}_{1},...,\mathbf{x}_{k}\right)
&:&=\sum_{j_{1},...,j_{k}=1}^{n_{1},...,n_{k}}p_{1j_{1}}...p_{kj_{k}}f\left(
\sum_{i=1}^{k}q_{i}x_{ij_{i}}\right) \\
&&-f\left( \sum_{i=1}^{k}q_{i}\sum_{j=1}^{n_{i}}p_{ij}x_{ij}\right) .
\end{eqnarray*}%
and the generalized\emph{\ Chebychev functional} by:%
\begin{eqnarray*}
&&\mathcal{T}_{k}\left( f,\mathbf{p}_{1},...,\mathbf{p}_{k},\mathbf{q},%
\mathbf{x}_{1},...,\mathbf{x}_{k}\right) \\
&=&\sum_{j_{1},...,j_{k}=1}^{n_{1},...,n_{k}}p_{1j_{1}}...p_{kj_{k}}%
\sum_{i=1}^{k}q_{i}\left( x_{ij_{i}}-\sum_{j=1}^{n_{i}}p_{ij}x_{ij}\right)
f\left( \sum_{i=1}^{k}q_{i}x_{ij_{i}}\right) .
\end{eqnarray*}
\end{definition}

We easily notice also that for $k=1$ this definition reduces to Definition %
\ref{main}. In \cite{mit10} the following estimation is obtained:

\begin{remark}
\label{7}If $f$ is a convex function then we have 
\begin{eqnarray*}
&&\min_{\substack{ 1\leq j_{1}\leq n_{1}  \\ ...  \\ 1\leq j_{k}\leq n_{k}}}%
\left\{ \frac{p_{1j_{1}}...p_{kj_{k}}}{r_{1j_{1}}...r_{kj_{k}}}\right\} 
\mathcal{J}_{k}\left( f,\mathbf{r}_{1},...,\mathbf{r}_{k},\mathbf{q},\mathbf{%
x}_{1},...,\mathbf{x}_{k}\right) \\
&\leq &\mathcal{J}_{k}\left( f,\mathbf{p}_{1},...,\mathbf{p}_{k},\mathbf{q},%
\mathbf{x}_{1},...,\mathbf{x}_{k}\right) \\
&\leq &\max_{\substack{ 1\leq j_{1}\leq n_{1}  \\ ...  \\ 1\leq j_{k}\leq
n_{k} }}\left\{ \frac{p_{1j_{1}}...p_{kj_{k}}}{r_{1j_{1}}...r_{kj_{k}}}%
\right\} \mathcal{J}_{k}\left( f,\mathbf{r}_{1},...,\mathbf{r}_{k},\mathbf{q}%
,\mathbf{x}_{1},...,\mathbf{x}_{k}\right) .
\end{eqnarray*}
\end{remark}

In this paper, we investigate upper and lower bounds that we have if the
function $f$ is superquadratic.\ 

Now we extend the earlier results. The following lemma describes the
behaviour of the functional under the superquadraticity condition:

\begin{lemma}
\label{1}Let $\mathbf{p}_{i},$ $\mathbf{x}_{i},$ $\mathbf{q}$ be as in
Definition \ref{112}. If $f$ is superquadratic then we have 
\begin{equation*}
\mathcal{J}_{k}\left( f,\mathbf{p}_{1},...,\mathbf{p}_{k},\mathbf{q},\mathbf{%
x}_{1},...,\mathbf{x}_{k}\right) \geq
\sum_{j_{1},...,j_{k}=1}^{n_{1},...,n_{k}}p_{1j_{1}}...p_{kj_{k}}f\left(
\left\vert \sum_{i=1}^{k}q_{i}x_{ij_{i}}-\bar{x}\right\vert \right) ,
\end{equation*}%
where $\bar{x}=\sum_{i=1}^{k}q_{i}\sum_{j=1}^{n_{i}}p_{ij}x_{ij}$ (we will
keep this notation throughout this subsection).
\end{lemma}

\begin{proof}
Straightforward from the definition of superquadratic functions.
\end{proof}

Using the same recipe as in the proof of Corollary \ref{N1}, we get:

\begin{corollary}
Let $\mathbf{p}_{i},$ $\mathbf{x}_{i},$ $\mathbf{q}$ be as in Definition \ref%
{112}. Let $f\geq 0$ be a superquadratic function defined on an interval $I=%
\left[ 0,a\right] $ or $\left[ 0,\infty \right) $, $\lambda \in \left[ 0,1%
\right] .$ Then 
\begin{equation*}
\mathcal{J}_{k}\left( f,\mathbf{p}_{1},...,\mathbf{p}_{k},\mathbf{q},\mathbf{%
x}_{1},...,\mathbf{x}_{k}\right) \geq
2\sum_{j_{1},...,j_{k}=1}^{n_{1},...,n_{k}}p_{1j_{1}}...p_{kj_{k}}f\left( 
\frac{1}{2}\left\vert \sum_{i=1}^{k}q_{i}x_{ij_{i}}-\bar{x}\right\vert
\right) .
\end{equation*}
\end{corollary}

The next result of the paper can be expressed as:

\begin{theorem}
\label{20}Let $f,$ $\mathbf{p}_{i},$ $\mathbf{x}_{i}\ $and $\mathbf{q\ }$be
as in Definition \ref{112} and the positive real numbers $r_{ij},$ $%
i=1,...,k\ $and $j=1,...,n_{i}$ such that $\sum_{j=1}^{n_{i}}r_{ij}=1$ for
all $i=1,...,k.$ We denote 
\begin{eqnarray*}
\mathbf{r}_{i} &=&\left( r_{i1},r_{i2},...,r_{in_{i}}\right) \ \text{\ for
all }i=1,...,k, \\
m &=&\min_{\substack{ 1\leq j_{1}\leq n_{1}  \\ ...  \\ 1\leq j_{k}\leq
n_{k} }}\left\{ \frac{p_{1j_{1}}...p_{kj_{k}}}{r_{1j_{1}}...r_{kj_{k}}}%
\right\} , \\
M &=&\max_{\substack{ 1\leq j_{1}\leq n_{1}  \\ ...  \\ 1\leq j_{k}\leq
n_{k} }}\left\{ \frac{p_{1j_{1}}...p_{kj_{k}}}{r_{1j_{1}}...r_{kj_{k}}}%
\right\} .
\end{eqnarray*}%
If $f$ is a superquadratic function, then:%
\begin{eqnarray*}
&&\mathcal{J}_{k}\left( f,\mathbf{p}_{1},...,\mathbf{p}_{k},\mathbf{q},%
\mathbf{x}_{1},...,\mathbf{x}_{k}\right) -m\mathcal{J}_{k}\left( f,\mathbf{r}%
_{1},...,\mathbf{r}_{k},\mathbf{q},\mathbf{x}_{1},...,\mathbf{x}_{k}\right)
\\
&\geq &mf\left( \left\vert \sum_{i=1}^{k}q_{i}\sum_{j=1}^{n_{i}}\left(
r_{ij}-p_{ij}\right) x_{ij}\right\vert \right) \\
&&+\sum_{j_{1},...,j_{k}=1}^{n_{1},...,n_{k}}\left(
p_{1j_{1}}...p_{kj_{k}}-mr_{1j_{1}}...r_{kj_{k}}\right) f\left( \left\vert
\sum_{i=1}^{k}q_{i}x_{ij_{i}}-\bar{x}\right\vert \right)
\end{eqnarray*}%
and%
\begin{eqnarray*}
&&M\mathcal{J}_{k}\left( f,\mathbf{r}_{1},...,\mathbf{r}_{k},\mathbf{q},%
\mathbf{x}_{1},...,\mathbf{x}_{k}\right) -\mathcal{J}_{k}\left( f,\mathbf{p}%
_{1},...,\mathbf{p}_{k},\mathbf{q},\mathbf{x}_{1},...,\mathbf{x}_{k}\right)
\\
&\geq &f\left( \left\vert \sum_{i=1}^{k}q_{i}\sum_{j=1}^{n_{i}}\left(
r_{ij}-p_{ij}\right) x_{ij}\right\vert \right) \\
&&+\sum_{j_{1},...,j_{k}=1}^{n_{1},...,n_{k}}\left(
Mr_{1j_{1}}...r_{kj_{k}}-p_{1j_{1}}...p_{kj_{k}}\right) f\left( \left\vert
\sum_{i=1}^{k}q_{i}x_{ij_{i}}-\bar{x}\right\vert \right) .
\end{eqnarray*}

\begin{proof}
We prove only the first inequality. Obviously%
\begin{eqnarray*}
&&\mathcal{J}_{k}\left( f,\mathbf{p}_{1},...,\mathbf{p}_{k},\mathbf{q},%
\mathbf{x}_{1},...,\mathbf{x}_{k}\right) -m\mathcal{J}_{k}\left( f,\mathbf{r}%
_{1},...,\mathbf{r}_{k},\mathbf{q},\mathbf{x}_{1},...,\mathbf{x}_{k}\right)
\\
&=&\sum_{j_{1},...,j_{k}=1}^{n_{1},...,n_{k}}\left(
p_{1j_{1}}...p_{kj_{k}}-mr_{1j_{1}}...r_{kj_{k}}\right) f\left(
\sum_{i=1}^{k}q_{i}x_{ij_{i}}\right) \\
&&+mf\left( \sum_{i=1}^{k}q_{i}\sum_{j=1}^{n_{i}}r_{ij}x_{ij}\right)
-f\left( \bar{x}\right) .
\end{eqnarray*}%
Since 
\begin{equation*}
\bar{x}=\sum_{j_{1},...,j_{k}=1}^{n_{1},...,n_{k}}\left(
p_{1j_{1}}...p_{kj_{k}}-mr_{1j_{1}}...r_{kj_{k}}\right)
\sum_{i=1}^{k}q_{i}x_{ij_{i}}+m\sum_{i=1}^{k}q_{i}%
\sum_{j=1}^{n_{i}}r_{ij}x_{ij}
\end{equation*}%
we conclude by Lemma \ref{1} that%
\begin{eqnarray*}
&&\mathcal{J}_{k}\left( f,\mathbf{p}_{1},...,\mathbf{p}_{k},\mathbf{q},%
\mathbf{x}_{1},...,\mathbf{x}_{k}\right) -m\mathcal{J}_{k}\left( f,\mathbf{r}%
_{1},...,\mathbf{r}_{k},\mathbf{q},\mathbf{x}_{1},...,\mathbf{x}_{k}\right)
\\
&\geq &\sum_{j_{1},...,j_{k}=1}^{n_{1},...,n_{k}}\left(
p_{1j_{1}}...p_{kj_{k}}-mr_{1j_{1}}...r_{kj_{k}}\right) f\left( \left\vert
\sum_{i=1}^{k}q_{i}x_{ij_{i}}-\bar{x}\right\vert \right) \\
&&+mf\left( \left\vert \sum_{i=1}^{k}q_{i}\sum_{j=1}^{n_{i}}r_{ij}x_{ij}-%
\bar{x}\right\vert \right)
\end{eqnarray*}%
\begin{eqnarray*}
&=&\sum_{j_{1},...,j_{k}=1}^{n_{1},...,n_{k}}\left(
p_{1j_{1}}...p_{kj_{k}}-mr_{1j_{1}}...r_{kj_{k}}\right) f\left( \left\vert
\sum_{i=1}^{k}q_{i}x_{ij_{i}}-\bar{x}\right\vert \right) \\
&&+mf\left( \left\vert \sum_{i=1}^{k}q_{i}\sum_{j=1}^{n_{i}}\left(
r_{ij}-p_{ij}\right) x_{ij}\right\vert \right) .
\end{eqnarray*}%
The proof of the other inequality goes likewise and we omit the details.
\end{proof}
\end{theorem}

The following particular case is of interest.

\begin{remark}
Let $\mathbf{p}_{1}=...=\mathbf{p}_{k}=\mathbf{p}$ and $\mathbf{x}_{1}=...=%
\mathbf{x}_{k}=\mathbf{x.}$ In this case we see that Lemma \ref{1} and
Theorem \ref{20} are recovering Proposition \ref{4}, respectively Theorem %
\ref{18} . Also for $k=1$ Lemma \ref{1} and Theorem \ref{20} recapture
Proposition \ref{6}, respectively Theorem \ref{5}.
\end{remark}

According to \cite[Lemma 2.2]{abr04}, if a superquadratic function is also
nonnegative then it is convex. We may conclude that in this particular case
Theorem \ref{20} is a refinement of the result stated in Remark \ref{7}.

Different results are obtained by using the Chebychev functional:

\begin{lemma}
\label{l}Let $f:\left[ 0,\infty \right) \rightarrow \mathbb{R}$. If there
exist real numbers $\tilde{m},$ $\tilde{M}$ such that $\tilde{m}\leq f\left(
\sum_{i=1}^{k}q_{i}x_{ij_{i}}\right) \leq \tilde{M},$ for all $j_{i}\in
1,...,n_{i},$ $i=1,...,k,$ then
\end{lemma}

\begin{equation}
\left\vert \mathcal{T}_{k}\left( f,\mathbf{p}_{1},...,\mathbf{p}_{k},\mathbf{%
q},\mathbf{x}_{1},...,\mathbf{x}_{k}\right) \right\vert \leq \frac{\tilde{M}-%
\tilde{m}}{2}%
\sum_{j_{1},...,j_{k}=1}^{n_{1},...,n_{k}}p_{1j_{1}}...p_{kj_{k}}\left\vert
\sum_{i=1}^{k}q_{i}x_{ij_{i}}-\bar{x}\right\vert .  \label{9}
\end{equation}

This lemma is a discrete version of a result due to P. Cerone and S. S.
Dragomir \cite[Theorem 2]{cer07}. See also \cite[Lemma 5.58]{dra03}.

\begin{proof}
Notice that 
\begin{equation*}
\left\vert f\left( \sum_{i=1}^{k}q_{i}x_{ij_{i}}\right) -\frac{\tilde{M}+%
\tilde{m}}{2}\right\vert \leq \frac{\tilde{M}-\tilde{m}}{2},
\end{equation*}%
for all $j_{i}\in 1,...,n_{i},$ $i=1,...,k.\ $Since%
\begin{equation*}
\sum_{j_{1},...,j_{k}=1}^{n_{1},...,n_{k}}p_{1j_{1}}...p_{kj_{k}}\left(
\sum_{i=1}^{k}q_{i}x_{ij_{i}}-\bar{x}\right) =0,
\end{equation*}%
we have%
\begin{eqnarray*}
&&\mathcal{T}_{k}\left( f,\mathbf{p}_{1},...,\mathbf{p}_{k},\mathbf{q},%
\mathbf{x}_{1},...,\mathbf{x}_{k}\right) \\
&=&\sum_{j_{1},...,j_{k}=1}^{n_{1},...,n_{k}}p_{1j_{1}}...p_{kj_{k}}\left(
\sum_{i=1}^{k}q_{i}x_{ij_{i}}-\bar{x}\right) \left( f\left(
\sum_{i=1}^{k}q_{i}x_{ij_{i}}\right) -\frac{\tilde{M}+\tilde{m}}{2}\right) ,
\end{eqnarray*}%
whence it follows that%
\begin{eqnarray*}
&&\left\vert \mathcal{T}_{k}\left( f,\mathbf{p}_{1},...,\mathbf{p}_{k},%
\mathbf{q},\mathbf{x}_{1},...,\mathbf{x}_{k}\right) \right\vert \\
&\leq
&\sum_{j_{1},...,j_{k}=1}^{n_{1},...,n_{k}}p_{1j_{1}}...p_{kj_{k}}\left\vert
\sum_{i=1}^{k}q_{i}x_{ij_{i}}-\bar{x}\right\vert \left\vert f\left(
\sum_{i=1}^{k}q_{i}x_{ij_{i}}\right) -\frac{\tilde{M}+\tilde{m}}{2}%
\right\vert \  \\
&\leq &\frac{\tilde{M}-\tilde{m}}{2}%
\sum_{j_{1},...,j_{k}=1}^{n_{1},...,n_{k}}p_{1j_{1}}...p_{kj_{k}}\left\vert
\sum_{i=1}^{k}q_{i}x_{ij_{i}}-\bar{x}\right\vert .
\end{eqnarray*}
\end{proof}

We close this subsection with a proposition that gives us an upper bound for
the Jensen functional under the superquadraticity condition, via the above
result on the Chebyshev functional.

\begin{proposition}
\label{13}Let $f:\left[ 0,\infty \right) \rightarrow \mathbb{R}$ be a
superquadratic function. If for $C(x)$ there exist real numbers $\tilde{m},$ 
$\tilde{M}$ such that $\tilde{m}\leq C\left(
\sum_{i=1}^{k}q_{i}x_{ij_{i}}\right) \leq \tilde{M},$ for all $j_{i}\in
1,...,n_{i},$ $i=1,...,k,$ then we have: 
\begin{eqnarray*}
&&\mathcal{J}_{k}\left( f,\mathbf{p}_{1},...,\mathbf{p}_{k},\mathbf{q},%
\mathbf{x}_{1},...,\mathbf{x}_{k}\right) \\
&\leq
&\sum_{j_{1},...,j_{k}=1}^{n_{1},...,n_{k}}p_{1j_{1}}...p_{kj_{k}}\left( 
\frac{\tilde{M}-\tilde{m}}{2}\left\vert \sum_{i=1}^{k}q_{i}x_{ij_{i}}-\bar{x}%
\right\vert -f\left( \left\vert \sum_{i=1}^{k}q_{i}x_{ij_{i}}-\bar{x}%
\right\vert \right) \right) .
\end{eqnarray*}
\end{proposition}

\begin{proof}
We apply 
\begin{eqnarray}
\mathcal{J}_{k}\left( f,\mathbf{p}_{1},...,\mathbf{p}_{k},\mathbf{q},\mathbf{%
x}_{1},...,\mathbf{x}_{k}\right) &\leq &\mathcal{T}_{k}\left( C,\mathbf{p}%
_{1},...,\mathbf{p}_{k},\mathbf{q},\mathbf{x}_{1},...,\mathbf{x}_{k}\right) 
\notag \\
&&-\sum_{j_{1},...,j_{k}=1}^{n_{1},...,n_{k}}p_{1j_{1}}...p_{kj_{k}}f\left(
\left\vert \sum_{i=1}^{k}q_{i}x_{ij_{i}}-\bar{x}\right\vert \right) .
\label{16}
\end{eqnarray}%
and the inequality (\ref{9}) in order to get the claimed result.

The proof of (\ref{16}) can be found in \cite[Theorem 3]{mitCIA10}.
\end{proof}

This proposition extends a result due to S. Abramovich and S. S. Dragomir 
\cite[Theorem 9]{abr09}. The inequality (\ref{16}) is establishing a
connection between the Jensen functional and the Chebychev functional and is
interesting in itself.

\subsection{The integral case}

In what follows we shall concentrate on the integral analogue of some of the
results from the previous section. Let $p_{i}\left( x\right) \mathrm{d}x$
and $r_{i}\left( x\right) \mathrm{d}x,$ $i=1,...,k$ be absolutely continuous
measures, where $p_{i},\ r_{i}:\left[ a,b\right] \subset \left( 0,\infty
\right) \rightarrow \left( 0,\infty \right) $ are such that $%
\int_{a}^{b}p_{i}\left( x\right) \mathrm{d}x=1$, $\int_{a}^{b}r_{i}\left(
x\right) \mathrm{d}x=1.$ We also consider $\mathbf{q}=\left(
q_{1},q_{2},...,q_{k}\right) ,$ $q_{i}>0$ with $\sum_{i=1}^{k}q_{i}=1.$ We
define\ 
\begin{equation*}
\mathcal{J}_{k}\left( f,p_{1},...,p_{k},\mathbf{q}\right) :=\int_{\left[ a,b%
\right] ^{k}}f\left( \sum_{i=1}^{k}q_{i}x_{i}\right) \prod_{i=1}^{k}\left(
p_{i}\left( x_{i}\right) \mathrm{d}x_{i}\right) -f\left(
\sum_{i=1}^{k}q_{i}\int_{a}^{b}xp_{i}\left( x\right) \mathrm{d}x\right)
\end{equation*}%
and%
\begin{equation*}
\mathcal{T}_{k}\left( f,p_{1},...,p_{k},\mathbf{q}\right) =\int_{\left[ a,b%
\right] ^{k}}\sum_{i=1}^{k}q_{i}\left( x_{i}-\int_{a}^{b}xp_{i}\left(
x\right) \mathrm{d}x\right) f\left( \sum_{i=1}^{k}q_{i}x_{i}\right)
\prod_{i=1}^{k}\left( p_{i}\left( x_{i}\right) \mathrm{d}x_{i}\right)
\end{equation*}%
for all positive integers $k$.

Before we prove the main result, we need the following lemma providing an
inequality that is interesting in itself as well. For the case of
superquadratic nonnegative functions (hence convex) this result is a
refinement of Jensen's inequality.

\begin{lemma}[the integral analogue of Lemma \protect\ref{1}]
\label{3}Assume that $f$ is superquadratic. Then 
\begin{equation}
\mathcal{J}_{k}\left( f,p_{1},...,p_{k},\mathbf{q}\right) \geq \int_{\left[
a,b\right] ^{k}}f\left( \left\vert \sum_{i=1}^{k}q_{i}x_{i}-\bar{x}%
\right\vert \right) \prod_{i=1}^{k}\left( p_{i}\left( x_{i}\right) \mathrm{d}%
x_{i}\right) ,  \label{in1}
\end{equation}%
where 
\begin{equation*}
\bar{x}=\sum_{i=1}^{k}q_{i}\int_{a}^{b}xp_{i}\left( x\right) \mathrm{d}x
\end{equation*}%
(we will keep this notation in this subsection).
\end{lemma}

\begin{proof}
Straightforward from the definition of superquadratic functions.
\end{proof}

\begin{example}
A particularly interesting case is\ pointed out by assuming, for simplicity,
that $p_{i}\left( x\right) \mathrm{d}x=\mathrm{d}x/\left( b-a\right) ,$ $%
i=1,...,k,$ when 
\begin{eqnarray*}
&&\frac{1}{\left( b-a\right) ^{k}}\int_{\left[ a,b\right] ^{k}}f\left(
\sum_{i=1}^{k}q_{i}x_{i}\right) \prod_{i=1}^{k}\mathrm{d}x_{i}-f\left( \frac{%
a+b}{2}\right) \\
&\geq &\frac{1}{\left( b-a\right) ^{k}}\int_{\left[ a,b\right] ^{k}}f\left(
\left\vert \sum_{i=1}^{k}q_{i}x_{i}-\frac{a+b}{2}\right\vert \right)
\prod_{i=1}^{k}\mathrm{d}x_{i}.
\end{eqnarray*}
\end{example}

\begin{remark}
\label{8}For the case $k=1$ the lemma gives us the following inequality 
\begin{equation*}
\int_{a}^{b}f\left( x\right) p\left( x\right) \mathrm{d}x\geq f\left(
\int_{a}^{b}xp\left( x\right) \mathrm{d}x\right) +\int_{a}^{b}f\left(
\left\vert x-\int_{a}^{b}xp\left( x\right) \mathrm{d}x\right\vert \right)
p\left( x\right) \mathrm{d}x
\end{equation*}%
for every $f$ superquadratic. This is the integral counterpart of
Proposition \ref{6}, an example of \cite[Theorem 2.3]{abr04}.
\end{remark}

We derive the following result.

\begin{theorem}[the integral analogue of Theorem \protect\ref{20}]
\label{th1} We denote%
\begin{equation*}
m=\inf_{t,s\in \left[ a,b\right] ;s\neq t}\left\{ \frac{\int_{\left[ t,s%
\right] ^{k}}\prod_{i=1}^{k}p_{i}\left( x_{i}\right) \mathrm{d}x_{i}}{\int_{%
\left[ t,s\right] ^{k}}\prod_{i=1}^{k}r_{i}\left( x_{i}\right) \mathrm{d}%
x_{i}}\right\}
\end{equation*}%
and%
\begin{equation*}
M=\sup_{t,s\in \left[ a,b\right] ;s\neq t}\left\{ \frac{\int_{\left[ t,s%
\right] ^{k}}\prod_{i=1}^{k}p_{i}\left( x_{i}\right) \mathrm{d}x_{i}}{\int_{%
\left[ t,s\right] ^{k}}\prod_{i=1}^{k}r_{i}\left( x_{i}\right) \mathrm{d}%
x_{i}}\right\} .
\end{equation*}%
If $f$ is superquadratic then 
\begin{eqnarray*}
&&\mathcal{J}_{k}\left( f,p_{1},...,p_{k},\mathbf{q}\right) -m\mathcal{J}%
_{k}\left( f,r_{1},...,r_{k},\mathbf{q}\right) \\
&\geq &mf\left( \left\vert \sum_{i=1}^{k}q_{i}\int_{a}^{b}x\left(
p_{i}\left( x\right) -r_{i}\left( x\right) \right) \mathrm{d}x\right\vert
\right) \\
&&+\int_{\left[ a,b\right] ^{k}}f\left( \left\vert \sum_{i=1}^{k}q_{i}x_{i}-%
\bar{x}\right\vert \right) \prod_{i=1}^{k}\left( \left( p_{i}\left(
x_{i}\right) -mr_{i}\left( x_{i}\right) \right) \mathrm{d}x_{i}\right)
\end{eqnarray*}%
and%
\begin{eqnarray*}
&&M\mathcal{J}_{k}\left( f,r_{1},...,r_{k},\mathbf{q}\right) -\mathcal{J}%
_{k}\left( f,p_{1},...,p_{k},\mathbf{q}\right) \\
&\geq &f\left( \left\vert \sum_{i=1}^{k}q_{i}\int_{a}^{b}x\left( p_{i}\left(
x\right) -r_{i}\left( x\right) \right) \mathrm{d}x\right\vert \right) \\
&&+\int_{\left[ a,b\right] ^{k}}f\left( \left\vert \sum_{i=1}^{k}q_{i}x_{i}-%
\bar{x}\right\vert \right) \prod_{i=1}^{k}\left( \left( Mr_{i}\left(
x_{i}\right) -p_{i}\left( x_{i}\right) \right) \mathrm{d}x_{i}\right) .
\end{eqnarray*}
\end{theorem}

\begin{proof}
We will prove the first inequality. Lemma \ref{3} implies that 
\begin{eqnarray*}
&&\mathcal{J}_{k}\left( f,p_{1},...,p_{k},\mathbf{q}\right) -m\mathcal{J}%
_{k}\left( f,r_{1},...,r_{k},\mathbf{q}\right) \\
&=&\int_{\left[ a,b\right] ^{k}}f\left( \sum_{i=1}^{k}q_{i}x_{i}\right)
\prod_{i=1}^{k}\left( \left( p_{i}\left( x_{i}\right) -mr_{i}\left(
x_{i}\right) \right) \mathrm{d}x_{i}\right) \\
&&+mf\left( \sum_{i=1}^{k}q_{i}\int_{a}^{b}xr_{i}\left( x\right) \mathrm{d}%
x\right) -f\left( \sum_{i=1}^{k}q_{i}\int_{a}^{b}xp_{i}\left( x\right) 
\mathrm{d}x\right)
\end{eqnarray*}%
\begin{eqnarray*}
&\geq &\int_{\left[ a,b\right] ^{k}}f\left( \left\vert
\sum_{i=1}^{k}q_{i}x_{i}-\bar{x}\right\vert \right) \prod_{i=1}^{k}\left(
\left( p_{i}\left( x_{i}\right) -mr_{i}\left( x_{i}\right) \right) \mathrm{d}%
x_{i}\right) \\
&&+mf\left( \left\vert \sum_{i=1}^{k}q_{i}\int_{a}^{b}x\left( p_{i}\left(
x\right) -r_{i}\left( x\right) \right) \mathrm{d}x\right\vert \right) .
\end{eqnarray*}%
The proof of the second inequality goes likewise and has been omitted.
\end{proof}

Now we turn our attention to the Chebychev functional. By an essentially
similar method as in the discrete case already discussed above, one can
prove the following lemma.

\begin{lemma}
We consider a superquadratic function $f:\left[ 0,\infty \right) \rightarrow 
\mathbb{R}$. If there exist real numbers $\tilde{m},$ $\tilde{M}$ such that $%
\tilde{m}\leq f(x)\leq \tilde{M},$ for all $x\geq 0,$ then we get
\end{lemma}

\begin{equation*}
\left\vert \mathcal{T}_{k}\left( f,p_{1},...,p_{k},\mathbf{q}\right)
\right\vert \leq \frac{\tilde{M}-\tilde{m}}{2}\int_{\left[ a,b\right]
^{k}}\left\vert \sum_{i=1}^{k}q_{i}x_{i}-\bar{x}\right\vert
\prod_{i=1}^{k}\left( p_{i}\left( x_{i}\right) \mathrm{d}x_{i}\right) .
\end{equation*}%
This lemma can be used to point out our last result.

\begin{proposition}[the integral analogue of Proposition \protect\ref{13}]
Using the above notations, we also consider a superquadratic function $f:%
\left[ 0,\infty \right) \rightarrow \mathbb{R}$. If there exist real numbers 
$\tilde{m},$ $\tilde{M}$ such that $\tilde{m}\leq C(x)\leq \tilde{M},$ for
all $x\geq 0,$ then we have: 
\begin{eqnarray*}
&&\mathcal{J}_{k}\left( f,p_{1},...,p_{k},\mathbf{q}\right) \\
&\leq &\int_{\left[ a,b\right] ^{k}}\left( \frac{\tilde{M}-\tilde{m}}{2}%
\left\vert \sum_{i=1}^{k}q_{i}x_{i}-\bar{x}\right\vert -f\left( \left\vert
\sum_{i=1}^{k}q_{i}x_{i}-\bar{x}\right\vert \right) \right)
\prod_{i=1}^{k}\left( p_{i}\left( x_{i}\right) \mathrm{d}x_{i}\right) .
\end{eqnarray*}
\end{proposition}

\begin{acknowledgement}
We would like to thank Professor S. S. Dragomir for suggesting the simpler
proof of Lemma \ref{l}.
\end{acknowledgement}

\end{document}